\documentclass[11pt, reqno]{amsart}
\usepackage{lmodern}
\usepackage[top=1.5in, left=1.2in, right=1.2in, bottom=1.5in]{geometry}
\usepackage{amssymb}
\usepackage{comment}
\usepackage{amsmath, amscd}
\usepackage{csquotes}
\usepackage{amsthm}
\usepackage[all]{xy}
\usepackage{color}
\usepackage{enumitem}
\usepackage{wrapfig}
\usepackage{graphicx}
\usepackage{hyperref}
\usepackage{tikz-cd}
\usepackage{mathtools}
\usepackage{caption}

\usepackage{palatino}

\newtheorem{theorem}{Theorem}[section]
\newtheorem*{theorem*}{Theorem}
\newtheorem{lemma}[theorem]{Lemma}
\newtheorem{corollary}[theorem]{Corollary}
\newtheorem{proposition}[theorem]{Proposition}

\theoremstyle{definition}
\newtheorem{definition}[theorem]{Definition}

\newtheorem{innercustomgeneric}{\customgenericname}
\providecommand{\customgenericname}{}
\newcommand{\newcustomtheorem}[2]{%
	\newenvironment{#1}[1]
	{%
		\renewcommand\customgenericname{#2}%
		\renewcommand\theinnercustomgeneric{##1}%
		\innercustomgeneric
	}
	{\endinnercustomgeneric}
}

\newcustomtheorem{customthm}{Theorem}

\begin{document}

\title{The Hausdorff and packing measure of some digital expansions}

\author{Daniel Ingebretson}

\begin{abstract}
We compute the exact Hausdorff and packing measure of sets of real numbers whose digital expansions in a given base are missing the digits beyond a given threshold.
\end{abstract}

\maketitle 

\section{Introduction}
Fix an integer $ l \geq 2 $.
Each $ x $ in the interval $ [0,1] $ has an infinite digital expansion 
$$ 
x = \sum_{k =1}^{\infty} \frac{a_k}{l^k} 
$$ 
in base $ l $ (also called an $ l $-ary expansion), where $ a_k \in \{0, 1, \ldots, l-1 \} $ are the digits of $ x $.
For each proper subset $ A \subset \{0,1,\ldots,l-1\} $, we may then consider the restricted digit set 
$$ 
C_A = \left\{ \sum_{k =1}^{\infty} \frac{a_k}{l^k} : a_k \in A \text{ for all } k \right\}.
$$
This set is invariant under the map $ x \mapsto lx \text{ (mod 1)} $ which preserves Lebesgue measure and is ergodic, so its Lebesgue measure is zero.

From classic results of Besicovitch \cite{Bes} and Eggleston \cite{Egg} (see also \cite{Bil}), the Hausdorff dimension of $ C_A $ is $ s = \log \# A / \log l $.
In fact, all notions of fractal dimension-- including the packing dimension-- coincide and equal this value.
Moreover, the $ s $-dimensional Hausdorff measure $ \mathcal{H}^s $ and packing measure $ \mathcal{P}^s $ of $ C_A $ are positive and finite (see e.g. Theorem 2.7 from \cite{Fal2}).

This leads to a natural question; what are the exact values of $ \mathcal{H}^s(C_A) $ and $ \mathcal{P}^s(C_A) $?
For general $ A $ this is open, but some cases are known.
For example, when $ l = 3 $ and $ A = \{0,2\} $ this is the usual ternary Cantor set, with dimension $ s = \log 2 / \log 3 $.
Falconer \cite{Fal} showed that the Hausdorff measure is equal to one, and Feng, Hua and Wen \cite{FHW} found that the packing measure is $ 4^s \simeq 2.398 $.
More generally, if the gaps in $ A $ are equally spaced (for example, when $ l = 9 $ and $ A = \{0,4,8\} $), the Hausdorff measure is equal to one.

In this paper, we will answer this question for the special case where $ A $ is the consecutive string $ \{0,1,\ldots,n-1\} $ with $ n < l $.

\begin{customthm}{1}
	\label{thm1}
	For integers $ l > n \geq 2 $, let $ C(n,l) $ be the set of all real numbers in the interval $ [0,1] $ whose base-$ l $ digital expansion digits are restricted to $ \{0,1,\ldots,n-1\} $.
	Then 
	$$ 
	\mathcal{H}^{s(n,l)}(C(n,l)) = \left(\frac{n-1}{l-1}\right)^{s(n,l)} \; \text{ and } \; \mathcal{P}^{s(n,l)}(C(n,l)) = 2^{s(n,l)}
	$$
	where $ s(n,l) = \log n / \log l $.
\end{customthm}

To estimate the Hausdorff measure, we study the maximal density of the self-similar measure supported on $ C(n,l) $. This was popularized by Ayer and Strichartz \cite{AS} and originates in the work of Marion \cite{Mar}, \cite{Mar2}.
An analagous link between the packing measure and the minimal density was established by Feng \cite{Feng}.

The proof for the Hausdorff measure formula is more technical than that of the packing measure formula.
The main reason for this is that intervals of maximal density need not exist, a phenomenon discovered by Ayer and Strichartz \cite{AS}.

On the other hand, Feng \cite{Feng} showed that there is always an interval realizing the minimal density.
To derive the packing measure formula, we use a variant of Feng's approach, while for the Hausdorff measure formula we develop some new methods which we hope can be used to calculate the Hausdorff measure of other self-similar sets as well.

\subsubsection*{Acknowledgements}
This research was made possible by a postdoctoral fellowship through the Kreitman School of Advanced Graduate Studies at Ben-Gurion University of the Negev.

\section{Self-similar sets in the interval}
\label{sss}
In this section, we will review some of the techniques that are commonly used to estimate the Hausdorff and packing measure of self-similar sets.
The sets we will be working with are attractors of finite linear iterated function systems (IFS for short) in the interval $ [0,1] $.
An IFS is a family $ \{ \phi_i : i \in I \} $ of contracting linear maps of $ [0,1] $ to itself, and its attractor is the unique non-empty compact set $ K \subset [0,1] $ satisfying
$$
K = \bigcup_{i \in I} \phi_i(K).
$$
We will further assume the \textit{open set condition}: that there exists an open $ U \subset [0,1] $ such that 
$$
\phi_i(U) \cap \phi_j(U) = \emptyset \text{ for } i \neq j.
$$ 
Under this assumption, if $ 0 < \lambda_i < 1 $ is the contraction coefficient of $ \phi_i $, the Hausdorff and packing dimensions of $ K $ coincide and are equal to the unique solution $ s $ of Moran's equation $ \sum_{i \in I} \lambda_i^s = 1 $.
By Hutchinson's work (see Section 4 from \cite{Hut} or Theorem 2.8 from \cite{Fal2}) there is a unique probability measure $ \mu $ supported on $ K $ satisfying the self-similarity relation
\begin{equation}
\label{selfsimmu}
\mu = \sum_{i \in I} \lambda_i^s \mu \circ \phi_i^{-1}.
\end{equation}

\subsection{The Hausdorff measure}
As mentioned in the introduction, the standard practice for estimating $ \mathcal{H}^s(K) $ is to relate it to the density of $ \mu $.
To make this connection precise, recall that for a given $ 0 \leq t < \infty $, the $ t $- pointwise density of a measure $ \nu $ supported on $ E \subset [0,1] $ at the point $ x \in [0,1] $ is
$$
D^{\ast t}(\nu, E, x) = \lim_{r \to 0} \sup \left\{ \frac{\nu(U)}{|U|^t} : U \subset [0,1] \text{ a closed interval, } \; x \in U, \; 0 < |U| \leq r \right\},
$$
using $ | \cdot | $ to denote diameter, which is length in this case.
A classical result from geometric measure theory says that if $ \mathcal{H}^t(E) < \infty $, then $ D^{\ast t}(\mathcal{H}^t |_E, E,x) = 1 $ for $ \mathcal{H}^t $-almost every $ x \in E $.
For a reference, consult Theorem 2.3 from \cite{Fal} or Remark 6.4(3) from \cite{Mat}.
Now, the sorts of self-similar sets we are considering satisfy $ \mathcal{H}^s(K) < \infty $, so
\begin{equation}
\label{Hdensity}
D^{\ast t}(\mathcal{H}^s |_K, K, x) = 1 
\end{equation}
for $ \mathcal{H}^s $-almost every $ x \in K $.

Another fact-- which follows from the scaling properties of Hausdorff measures and uniqueness of the measure satisfying Equation \ref{selfsimmu}-- is that the self-similar measure coincides with the normalized Hausdorff measure restricted to $ K $, i.e.
\begin{equation}
\label{Hequalsmu}
\mu = \frac{1}{\mathcal{H}^s(K)} \mathcal{H}^s |_K.
\end{equation}
This is equality, not equivalence of measures.
By combining Equations \ref{Hdensity} and \ref{Hequalsmu}, we obtain that
\begin{equation}
\label{Hdensity2}
\mathcal{H}^s(K) \lim_{r \to 0} \sup \left\{ \frac{\mu(U)}{|U|^s} : U \subset [0,1] \text{ a closed interval, } \; x \in U, \; 0 < |U| \leq r \right\} = 1
\end{equation}
for $ \mu $-almost every $ x \in K $.

We now introduce the notation 
$$ 
d(U) = \frac{\mu(U)}{|U|^s}
$$
for the density of $ \mu $ with respect to a set $ U \subset [0,1] $, which in future sections will not necessarily be an interval.
The \textit{maximal density} of $ \mu $ is
\begin{equation}
\label{dmax}
d_{\max} = \sup \{ d(U) : U \subset [0,1] \text{ a closed interval} \}.
\end{equation}
It is not difficult to show that
$$
d_{\max} = \lim_{r \to 0} \sup \{ d(U) : U \subset [0,1] \text{ a closed interval, } \; x \in U, \; 0 < |U| \leq r \}
$$
for $ \mu $-almost every $ x \in K $, so that 
\begin{equation}
\label{dmaxinv}
\mathcal{H}^s(K) = d_{\max}^{-1}
\end{equation}
by Equation \ref{Hdensity2}.
This is the aforementioned connection between the Hausdorff measure of $ K $ and the density of $ \mu $.

Unfortunately the maximal density is difficult to compute in practice, because the class of all closed subintervals of $ [0,1] $ is too large.
It is natural then to search for a more manageable subclass with the same supremum, which was accomplished by Marion (\cite{Mar}, Th\'{e}or\`{e}me 6.1).
First, a \textit{basic interval at level} $ k $ is an interval of the form
$$
\phi_{\omega_1} \circ \cdots \circ \phi_{\omega_k}([0,1]) 
$$
for some $ \omega_1 \cdots \omega_k \in I^k $.
Letting $ \mathfrak{B}_k $ be the collection of all unions of level-$k$ basic intervals, Marion's result is that
$$
d_{\max} = \sup \{ d(U) : U \in \mathfrak{B}_k, \; k \geq 1 \}.
$$

In fact, the computation of maximal density can be simplified bit further.
By Equation \ref{selfsimmu}, the measure of a basic interval at level $ k $ is
\begin{equation}
\label{selfsimint}
\mu(\phi_{\omega_1} \circ \cdots \circ \phi_{\omega_k}([0,1])) = \lambda_{\omega_1}^s \cdots \lambda_{\omega_k}^s.
\end{equation}
One consequence of this is that the density enjoys the invariance property $ d(U) = d(\phi_i^{-1}(U)) $ when $ U \subset \phi_i([0,1]) $ for some $ i \in I $.
This was termed the ``blow-up principle" in \cite{AS}, and will be used often in subsequent proofs.
For now, notice that if $ U \in \mathfrak{B}_k $ and $ i \in I $ are arbitrary, then $ \phi_i(U) \in \mathfrak{B}_{k+1} $ has the same density as $ U $ by this principle.
As a result, $ \max \{d(U) : U \in \mathfrak{B}_k\} $ is nondecreasing in $ k $, and 
\begin{equation}
\label{dmaxlim}
d_{\max} = \lim_{k \to \infty} \max \{ d(U) : U \in \mathfrak{B}_k \}.
\end{equation}
This equation will serve as a point of departure for the proof of the Hausdorff measure formula for $ C(n,l) $ from Theorem \ref{thm1}.

\subsection{The packing measure}
The analogue of Equation \ref{dmaxinv} for the packing measure is
\begin{equation}
\label{dmcinv}
\mathcal{P}^s(K) = d_{\text{mc}}^{-1},
\end{equation}
where
\begin{equation}
\label{dmc1}
d_{\text{mc}} = \inf \{ d(U) : U \subset [0,1] \text{ a closed interval centered in } K \}
\end{equation}
is the \textit{minimal centered density} of $ \mu $.
The proof, due to Feng (\cite{Feng}, Corollary 2.3), is quite similar to the proof of Equation \ref{dmaxinv} sketched in the previous section. 
The main difference is that the restriction that the intervals are centered in $ K $ is already strong enough to ensure that the infimum in Equation \ref{dmc1} is realized, while the supremum in Equation \ref{dmax} need not be.

Indeed, in our proof of the Hausdorff measure formula from Theorem \ref{thm1} given in Section \ref{Hmeas}, the maximal density is not realized by any union of basic intervals.
Rather, it must instead be approximated from below by an appropriate sequence of intervals, which are then shown to have maximal density in a given class.

In contrast, to compute the packing measure in Section \ref{Pmeas} we simply exhibit an interval realizing the minimal centered density, and then show that there are no other intervals of lower centered density.
This is the reason that the packing measure of a given self-similar set is often easier to compute than the Hausdorff measure, despite the fact that the Hausdorff measure was defined much earlier, and that the definition and properties of the packing measure are more complicated.

In the same paper \cite{Feng}, Feng gave a formula for the minimal centered density, under the additional hypotheses that maps $ \phi_i $ are all orientation-preserving, and that the self-similar set in question contains the endpoints of $ [0,1] $.
The set $ C(n,l) $ that we are interested in is missing the right endpoint, so we will need to modify Feng's approach.

\section{The Hausdorff measure of $ C(n,l) $}
\label{Hmeas}
Fix integers $ l > n \geq 2 $, and set $ I = \{0,1,\ldots,n-1\} $.
The first $ n $ inverse branches of $ l \mapsto lx \text{ (mod 1)} $ are
$ \phi_i(x) = (x+i)/l $ with $ i \in I $, and together they generate an iterated function system whose attractor is $ C(n,l) $.
Explicitly $ C(n,l) = \bigcap_{k \geq 1} O_k $, where
\begin{equation}
\label{Ok}
O_k = \bigcup_{\omega \in I^k} \phi_{\omega_1} \circ \cdots \circ \phi_{\omega_k}([0,1]).
\end{equation}

Let $ s = \log n / \log l $ be the (packing and Hausdorff) dimension, $ \mu $ the associated self-similar measure supported on $ C(n,l) $, and $ d(U) = \mu(U)/|U|^s $ the density of a set $ U \subset [0,1] $ with respect to the measure $ \mu $.
Since the contraction coefficients of $ \phi_i $ are all equal to $ l^{-1} $, Equation \ref{selfsimint} takes the simple form
$$
\mu(\phi_{\omega_1} \circ \cdots \circ \phi_{\omega_k}([0,1])) = n^{-k}
$$
independent of $ \omega_1 \cdots \omega_k \in I^k $.

Set $ r = \sup C(n,l) $ and notice that $ |C(n,l)| = r $ because $ 0 \in C(n,l) $.
As the fixed point of $ \phi_{n-1} $, it is easy to see that $ r = (n-1)/(l-1) $, so the Hausdorff measure formula from Theorem \ref{thm1} can be re-stated as $ \mathcal{H}^s(C(n,l)) = r^s $, which, in view of Equation \ref{dmaxinv}, is equivalent to $ d_{\max} = r^{-s} $.

\vspace{0.2cm}
\noindent \textit{Remark.}
All of the quantities $ \phi_i $, $ I $, $ s $, $ \mu $, and $ r $, as well as quantities yet to be defined in terms of them, depend on $ n $ and $ l $.
For the purpose of avoiding clumsy notation, we are suppressing these dependences.
\vspace{0.2cm}

The set $ O_k $ defined in Equation \ref{Ok} has maximal diameter among elements of $ \mathfrak{B}_k $, and the following proposition shows that it has maximal density as well.
\begin{proposition}
	\label{Hpropdec}
	For all $ k \geq 1 $, $ \max \{d(U) : U \in \mathfrak{B}_k\} = d(O_k) $.
\end{proposition}

The proof of this proposition will occupy the rest of this section, but first let us show why $ d_{\max} = r^{-s} $ follows from it.
It is standard that $ O_k \to C(n,l) $ in the Hausdorff metric on $ [0,1] $, so in particular $ |O_k| \rightarrow r $.
Since $ \mu(O_k)=1 $ for all $ k $, we have $ d(O_k) = |O_k|^{-s} $, and therefore $ d_{\max} = \lim_{k \to \infty} |O_k|^{-s} = r^{-s} $ using Equation \ref{dmaxlim}.

To prove Proposition \ref{Hpropdec}, we will first rule out some subclasses of $ \mathfrak{B}_k $ whose densities are not maximal. 
To handle the remaining cases, we will derive formulas for the lengths of gap intervals at various levels, and use these to relate the density of an arbitrary element of $ \mathfrak{B}_k $ to densities of highly symmetric elements we call \textit{clusters}, which we now define.

\subsection{Clusters and gaps}

\begin{definition}
	\label{clustdef}
	For $ 0 \leq i \leq k $, clusters of type $ (i,k) $ are defined as follows.
	\begin{itemize}
		\item The clusters of type $ (0,k) $ are the level-$k$ basic intervals.
		\vspace{0.1cm}
		\item For $ 1 \leq i \leq k-1 $, the clusters of type $ (i,k) $ are 
		$$
		\bigcup_{\omega_1 \cdots \omega_i \in I^i} \phi_{\tau_1} \circ \cdots \circ  \phi_{\tau_{k-i}} \circ \phi_{\omega_1} \circ \cdots \circ \phi_{\omega_i}([0,1])
		$$
		for some fixed $ \tau_1 \cdots \tau_{k-i} \in I^{k-i} $.
		\vspace{0.1cm}
		\item The only cluster of type $ (k,k) $ is $ O_k $ as defined in Equation \ref{Ok}.
	\end{itemize}
\end{definition}

In other words, a cluster of type $ (i,k) $ is a union of basic intervals coded by words with a common prefix of length $ k-i $.

\begin{definition}
	\label{gapdef}
	The gap intervals at level $ k $ are the connected components of
	$$ 
	[0,1] \setminus \bigcup_{\omega \in I^k} \phi_{\omega_1} \circ \cdots \circ \phi_{\omega_k}([0,1]).
	$$ 
\end{definition}

We will use the clusters to classify the level-$ k $ gap intervals by their lengths, as follows.
The largest level-$ k $ gap interval lies between $ O_k $ and the right endpoint of $ [0,1] $; we will call this the \textit{gap of type} $ (k,k) $.
Within $ O_k $ there are $ n $ clusters of type $ (k-1,k) $, which are separated by $ n-1 $ gap intervals of equal length, called gaps of type $ (k-1,k) $.
Within each type-$ (k-1,k) $ cluster lie $ n $ clusters of type $ (k-2,k) $, which are also separated by equal length gap intervals, called type-$ (k-2,k) $ gaps.
Continuing in this way, we come to the type-$ (1,k) $ clusters, which within each type-$ (2,k) $ cluster are separated by equal length type-$ (1,k) $ gaps.
Each of these type-$ (1,k) $ clusters consist of $ n $ basic intervals, i.e. type-$ (0,k) $ clusters, whose endpoints are overlapping.
To complete the hierarchy of gaps, we can think of these as being ``separated" by ``gaps" of length zero; these are the type- $ (0,k) $ gaps.

\vspace{0.2cm}
\noindent \textit{Remark.} It is \textit{not} true that all the type-$ (i,k) $ clusters are separated by type-$ (i,k) $ gaps.
Rather, they are separated by an assortment of gaps of type $ (i,k), (i+1,k), \ldots, (k-1,k) $, but none of lower type than $ (i,k) $.
What \textit{is} true is that the type-$ (i,k) $ clusters \textit{that lie within a single type}-$ (i+1,k) $ \textit{cluster} are separated by type-$ (i,k) $ gaps.
\vspace{0.2cm}

\begin{lemma}
	\label{gaplem}
	For $ k \geq 1 $ and $ 1 \leq i \leq k $ the length of a gap of type $ (i,k) $ is 
	$$
	\frac{l-n}{l^k}(1+l+l^2+\cdots+l^{i-1}).
	$$
\end{lemma}

\begin{proof}
	We proceed by induction on $ k $.
	For the base case $ k=1 $ there is only one gap interval, the type-$ (1,1) $ gap between $ O_1 $ and the right endpoint of $ [0,1] $, with length $ 1-\phi_{n-1}(1) = (l-n)/l $.
	
	Now assume the result for $ k $.
	By definition, for $ i \leq k $ the gaps of type $ (i,k+1) $ are the intervals between clusters of type $ (i,k+1) $ that lie in a single cluster of type $ (i+1, k+1) $.
	Fix such a cluster
	$$
	A = \bigcup_{\omega_1 \cdots \omega_{i+1} \in I^{i+1}} \phi_{\tau_1} \circ \cdots \circ \phi_{\tau_{k-i}} \circ \phi_{\omega_1} \circ \cdots \circ \phi_{\omega_{i+1}}([0,1]),
	$$
	and consider the type-$ (i+1,k) $ cluster $ \phi_{\tau_1}^{-1}(A) $.
	This contains $ n $ type-$ (i,k) $ clusters
	$$
	\bigcup_{\omega_2 \cdots \omega_{i+1} \in I^i} \phi_{\tau_2} \circ \cdots \circ \phi_{\tau_{k-i}} \circ \phi_{\omega_1} \circ \phi_{\omega_2} \circ \cdots \circ \phi_{\omega_{i+1}}([0,1]), \quad \omega_1 = 0,\ldots,n-1,
	$$
	that lie in a common type-$ (i+1,k) $ cluster, namely $ \phi_{\tau_1}^{-1}(A) $, so the $ n-1 $ gaps that lie between them are of type $ (i,k) $. 
	By our inductive hypothesis, these gaps have equal length $ l^{-k} (l-n)(1+l+l^2+\cdots+l^{i-1}) $.
	Applying the map $ \phi_{\tau_1} $, which has contraction coefficient $ l^{-1} $, we then have that the $ n $ type-$ (i, k+1) $ clusters
	$$
	\bigcup_{\omega_1 \cdots \omega_{i+1} \in I^i} \phi_{\tau_1} \circ \cdots \circ \phi_{\tau_{k-i}} \circ \phi_{\omega_1} \circ \phi_{\omega_2} \circ \cdots \circ \phi_{\omega_{i+1}}([0,1]), \quad \omega_1 = 0,\ldots,n-1,
	$$
	are separated by gap intervals of equal length $ l^{-(k+1)} (l-n)(1+l+l^2+\cdots+l^{i-1}) $.
	Since they lie in a common cluster of type $ (i+1,k+1) $, namely $ A $, these are indeed gaps of type $ (i,k+1) $.
\end{proof}

\begin{lemma}
\label{clustlem}
For $ k \geq 2 $ and $ 2 \leq i \leq k $ the diameter of a cluster of type $ (i,k) $ is
$$
\frac{1}{l^k}\big(n+(n-1)(l+l^2+\cdots+l^{i-1})\big),
$$
\end{lemma}

\begin{proof}
Fix $ k \geq 2 $.
For $ i=2 $, a cluster of type $ (2,k) $ consists of $ n $ clusters of type $ (1,k) $ separated by $ n-1 $ gaps of type $ (1,k) $.
Each cluster of type $ (1,k) $ consists of $ n $ contiguous level-$k$ intervals, so using Lemma \ref{gaplem}, the diameter of a type-$ (2,k) $ cluster is
$$
\frac{1}{l^k}(n^2+(n-1)(l-n)) = \frac{1}{l^k}(n+(n-1)l),
$$
and the formula holds for $ i=2 $.

To complete the proof, we will assume the result for some $ i \in \{2, \ldots, k-1\} $, and show that it holds for $ i+1 $.
For convenience, we will use the notation $ a_i = l+l^2+\cdots+l^{i-1} $.
It follows from $ l(1+a_i)=a_i+l^i $ that
\begin{equation}
\label{h1}
n+(n-1)a_i+(l-n)(1+a_i) = l^i.
\end{equation}
The diameter of a cluster of type $ (i+1,k) $ is the sum of the diameters of its $ n $ clusters of type $ (i,k) $ and the lengths of its $ n-1 $ type-$ (i,k) $ gaps.
The first is given by our assumption, and the second from the formula in Lemma \ref{gaplem}.
Adding these quantities and simplifying using Equation \ref{h1} yields
\begin{multline*}
\frac{1}{l^k}\big( n(n+(n-1)a_i) + (n-1)(l-n)(1+a_i) \big) \\
= \frac{1}{l^k} \big( n+(n-1)a_i + (n-1)(n+(n-1)a_i+(l-n)(1+a_i)) \big) \\
= \frac{1}{l^k} \big( n+(n-1)a_i + (n-1)l^i \big) = \frac{1}{l^k} \big( n+(n-1)(a_i+l^i) \big), \hspace{2.7cm}
\end{multline*}
so the formula holds for $ i+1 $.
\end{proof}

\subsection{Density of consecutive clusters}
A union of clusters of a common type is called \textit{consecutive} if it contains all the clusters of this type between its leftmost and rightmost cluster.
For each $ 1 \leq i \leq k $, let $ \mathfrak{C}_i^{(k)} \subset \mathfrak{B}_k $ be the collection of all consecutive unions of clusters of type $ (i,k) $.

The following proposition is the main technical result of this section.

\begin{proposition}
	\label{mainhlem}
	For $ k \geq 1 $ and $ i \in \{0,\ldots,k-1\} $, if $ U \in \mathfrak{C}_i^{(k)} $ is a union of at least two clusters of type $ (i,k) $, and $ U' $ is cluster of type $ (i,k) $ that is separated from $ U $ by a gap of type $ (i,k) $, then $ d(U) \leq d(U \cup U') $.
\end{proposition}

\begin{proof}
	Since $ |U'| < |U| < |U \cup U'| $, setting
	\begin{equation}
	\label{lambda}
	\lambda = \frac{|U \cup U'|-|U|}{|U \cup U'|-|U'|}
	\end{equation}
	we have $ 0 < \lambda < 1 $ and
	$$
	|U| = \lambda |U'| + (1-\lambda)|U \cup U'|.
	$$
	Since $ t \mapsto t^s $ is concave,
	$$
	|U|^s \geq \lambda |U'|^s + (1-\lambda) |U \cup U'|^s = |U \cup U'|^s - \lambda (|U \cup U'|^s - |U'|^s),
	$$
	and thus
	\begin{multline*}
	d(U \cup U') \geq \frac{\mu(U)+\mu(U')}{|U|^s+\lambda(|U \cup U'|^s - |U'|^s)}
	\geq \min \left\{ \frac{\mu(U)}{|U|^s}, \frac{\mu(U')}{\lambda(|U \cup U'|^s - |U'|^s)} \right\} \\ 
	= \min \left\{ d(U), \frac{\lambda^{-1} \mu(U')}{(|U \cup U'|^s - |U'|^s)} \right\},
	\end{multline*}
	so to deduce the desired result it suffices to show that
	$$
	d(U) \leq \frac{\lambda^{-1}\mu(U')}{(|U \cup U'|^s - |U'|^s)}.
	$$
	Suppose that $ U $ consists of $ p \geq 2 $ clusters of type $ (i,k) $.
	Then $ \mu(U) = p \mu(U') $ and the above equation is equivalent to
	$$
	p |U \cup U'|^s \leq \lambda^{-1} |U|^s + p|U'|^s.
	$$
	Because $ s < 1 $, we have $ |U \cup U'|^s \leq |U|^s + |U'|^s $, so if we can show that $ p \leq \lambda^{-1} $ then we are done.
	We will show this separately for the following three cases.
	
	\textit{Case 1: $ i=0 $.}
	In this case, $ U $ and $ U' $ overlap at an endpoint (since gaps of type $ (0,k) $ have zero length), so $ \lambda = |U'|/|U| $ by Equation \ref{lambda}.
	If the gap intervals in $ U $ sum to $ q/l^k $ for some integer $ q \geq 0 $, then $ |U| = (p+q)/l^k $ and $ |U'| = 1/l^k $, so that $ \lambda = 1/(p+q) \leq p^{-1} $.
	
	\textit{Case 2: $ i=1 $.}
	The $ p $ clusters of type $ (1,k) $ comprising $ U $ are separated by $ p-1 $ gaps of various types from the list $ (1,k), (2,k), \ldots, (k-1,k) $ so by Lemma \ref{gaplem} their sum can be written as
	$$
	\frac{l-n}{l^k}(p-1+N)
	$$
	for some nonnegative integer $ N $.
	Because a cluster of type $ (1,k) $ consists of $ n $ level-$ k $ basic intervals with no gaps, its diameter is $ n/l^k $. Together, this implies that
	$$
	l^k|U| = pn+(l-n)(p-1+N).
	$$
	By our assumption $ U $ and $ U' $ are separated by a gap of type $ (1,k) $, whose length given in Lemma \ref{gaplem} is $ (l-n)/l^k $.
	As a result,
	$$
	l^k(|U \cup U'| - |U'|) = l^k|U|+l-n = pn+(l-n)(p+N)  = l(p+(l-n)N)
	$$
	and
	$$
	l^k(|U \cup U'| - |U|) = l^k|U'|+l-n = n+l-n = l,
	$$
	so $ \lambda = 1/(p+(l-n)N) \leq p^{-1} $ using Equation \ref{lambda}.
	
	\textit{Case 3: $ i \geq 2 $.}
	Recall the notation $ a_i = l+l^i+\cdots+l^{i-1} $ from the proof of Lemma \ref{clustlem}.
	We will use this and Equation \ref{h1} in the following calculations.
	
	By assumption, $ U $ consists of $ p $ clusters of type $ (i,k) $ that are separated by $ p-1 $ gaps of various types from the list $ (i,k), (i+1,k), \ldots, (k-1,k) $.
	By Lemma \ref{gaplem} the lengths of each of these gaps have the form
	$$
	\frac{l-n}{l^k}(1+a_i+N_j l^i), \quad j=1,\ldots,p-1
	$$
	where $ N_j $ are nonnegative integers.
	Set $ N = N_1 + \cdots + N_{p-1} $.
	Summing the diameters of the $ p $ clusters (using Lemma \ref{clustlem}) and the lengths of the $ p-1 $ gaps then gives
	\begin{align*}
	l^k|U| &= p(n+(n-1)a_i) + (p-1)(l-n)(1+a_i) + (l-n)l^i N \\
	&= n+(n-1)a_i + (p-1)(n+(n-1)a_i+(l-n)(1+a_i)) + (l-n)l^i N \\
	&= n+(n-1)a_i + (p-1)l^i + (l-n)l^i N,
	\end{align*}
	using Equation \ref{h1} in the last equality.
	Because $ U' $ is separated from $ U $ by a gap of type $ (i,k) $, we then have
	\begin{align*}
	l^k(|U \cup U'| - |U'|) &= l^k|U| + (l-n)(1+a_i) \\
	&= n+(n-1)a_i + (p-1)l^i + (l-n)(1+a_i+l^i N) \\
	&= n+(n-1)a_i + (l-n)(1+a_i) + (p-1)l^i + (l-n) l^i N \\
	&= l^i + (p-1)l^i + (l-n)l^i N \\
	&= l^i (p+(l-n)N),
	\end{align*}
	again using Equation \ref{h1} in the second to last equality.
	Since $ U' $ is a cluster of type $ (i,k) $, by adding the formulas from Lemmas \ref{gaplem} and \ref{clustlem} we obtain
	$$
	l^k(|U \cup U'| - |U|) = n+(n-1)a_i+(l-n)(1+a_i) = l^i,
	$$
	in the same way, so by Equation \ref{lambda}, $ \lambda = 1/(p+(l-n)N) \leq p^{-1} $.
\end{proof}

We are now in a position to prove Proposition \ref{Hpropdec}.
The proof will use the terminology of ``endpoints" of a given $ U \in \mathfrak{C}_i^{(k)} $, by which we mean the left endpoint of the leftmost basic interval in $ U $, and the right endpoint of the rightmost basic interval.

\begin{proof}[Proof of Proposition \ref{Hpropdec}.]
	Since $ O_k \in \mathfrak{B}_k $, the ``$\geq$" inequality is immediate.
	We will prove the reverse inequality by induction on $ k $.
	The base case $ k=1 $ is trivial since there is only one cluster of type $ (1,1) $, namely $ O_1 $.
	Assume the result for $ k $, and take $ U \in \mathfrak{B}_{k+1} $.
	if $ U $ is not consecutive, by filling in the missing basic intervals in $ U $ the density strictly increases, so we may assume that $ U $ is consecutive.
	
	First, consider the case that $ U $ does not contain a gap of type $ (k,k+1) $.
	The only gaps of type $ (k,k+1) $ are those separating the clusters of type $ (k,k+1) $, so $ U $ must lie entirely inside some type-$ (k,k+1) $ cluster, say
	$$
	U \subset \bigcup_{\omega_1 \cdots \omega_k \in I^k} \phi_i \circ \phi_{\omega_1} \circ \cdots \circ \phi_{\omega_k}([0,1])
	$$
	for some $ i \in \{0,1,\ldots,n-1\} $.
	By the blow-up principle $ d(U) = d(\phi_i^{-1}(U)) $, and since $ \phi_i^{-1}(U) \in \mathfrak{B}_k $, we have $ d(U) \leq d(O_k) $ by our inductive hypothesis.
	Because $ \mu(O_k) = 1 $ and $ |O_{k+1}| < |O_k| $ for all $ k $, it follows that $ d(O_k) < d(O_{k+1}) $ and thus $ d(U) < d(O_{k+1}) $, which concludes the proof for this case.
	
	Now suppose that $ U $ contains a gap of type $ (k,k+1) $.
	At least one of the endpoints of $ U $ borders a gap, and as long as $ U \neq O_{k+1} $, it borders a gap of type $ (i_1,k+1) $ for some $ 1 \leq i_1 \leq k $, where $ i_1 $ is chosen minimal.
	Note that we are allowing the possility that $ i_1 = 0 $, i.e. that the gap has zero length.
	Then $ U \in \mathfrak{C}_{i_1}^{(k+1)} $, and since $ U $ contains a type-$ (k,k+1) $ gap, $ U $ contains at least two type-$ (i_1,k) $ clusters.
	If $ U_1 $ is the cluster of type $ (i_1,k) $ that is separated from $ U $ by the type-$ (i_1,k+1) $ gap, by Lemma \ref{mainhlem}, $ d(U) \leq d(U \cup U_1) $.
	
	We now repeat the process: if $ U \cup U_1 $ borders a gap of minimal type $ (i_2,k+1) $, let $ U_2 $ be the type-$ (i_2,k+1) $ cluster that is separated from $ U \cup U_1 $ by this gap, so $ d(U \cup U_1) \leq d(U \cup U_1 \cup U_2) $ again by Lemma \ref{mainhlem}.
	
	This process terminates at a finite stage $ j $ where the endpoints of $ U \cup U_1 \cup \cdots \cup U_j $ do not border any gap of type $ (i,k+1) $ for any $ 1 \leq i \leq k $.
	This is only possible if $ U \cup U_1 \cup \cdots \cup U_j $ is a cluster of maximal type $ (k+1,k+1) $, that is, $ O_{k+1} $.	
	Combining these inequalities gives
	$$
	d(U) \leq d(U \cup U_1) \leq \cdots \leq d(U \cup U_1 \cup \cdots \cup U_j) = d(O_{k+1}).
	$$
\end{proof}

\section{The packing measure of $ C(n,l) $}
\label{Pmeas}
By Equation \ref{dmcinv}, it suffices to show that $ d_{\text{mc}} = 2^{-s} $.
First, a slight improvement to Equation \ref{dmc}: if $ U \subset [0,1] $ is contained in a level-one basic interval, by repeatedly applying the blow-up principle we eventually obtain an interval with the same density that is not contained in any level-one basic interval, so
\begin{multline}
\label{dmc}
d_{\text{mc}} = \inf \left\{ d(U) : U \subset [0,1] \text{ a closed interval centered in } C(n,l) \right. \\
\left. \text{ with } U \not \subset \phi_i([0,1]) \text{ for any } i \in I \right\}.
\end{multline}

To prove that $ d_{\text{mc}} \leq 2^{-s} $, is not difficult to find intervals centered in $ C(n,l) $ with density $ 2^{-s} $.
The trick is to choose intervals half inside $ C(n,l) $ and half outside.
For instance, consider $ [r-l^{-1},r+l^{-1}] $.
Its right half has $ \mu $-measure zero, and the measure of its left half is
$$
\mu([r-l^{-1},r]) = \mu([\phi_{n-2}(r), \phi_{n-1}(r)]) = \mu([\phi_{n-1}(0),\phi_{n-1}(1)]) = n^{-1},
$$
so its density is $ n^{-1}/(2l^{-1})^s = 2^{-s} $.

\begin{lemma}
	\label{packlem}
	$ d([0,y]) \geq 1 $ for all $ 0 < y \leq r $, and $ d([x,r]) \geq 1 $ for all $ 0 \leq x < r $.
\end{lemma}

\begin{proof}
	The proof is a variant of Lemma 3.1 from \cite{Feng}.
	Let $ D_1 = \inf \{ d([0,y]) : 0 < y \leq r \} $.
	Because $ y \mapsto d([0,y]) $ is continuous, there exists $ y_0 $ such that $ D_1 = d([0,y_0]) $.
	Certainly $ y_0 $ is in some level-one basic interval, say $ \phi_i([0,1]) $, and by Equation \ref{dmc} we may assume that $ i \geq 1 $.
	We claim that $ y_0 = \phi_i(0) $.
	
	Suppose by contradiction that $ y_0 > \phi_i(0) $, and let $ u = y_0-\phi_i(0) $.
	Then
	\begin{multline*}
	D_1 = \frac{\mu([0,\phi_i(0)])+\mu([\phi_i(0),\phi_i(0)+u])}{(\phi_i(0)+u)^s}
	\geq \frac{\mu([0,\phi_i(0)])+\mu([\phi_i(0),\phi_i(0)+u])}{\phi_i(0)^s + u^s} \\
	\geq \min \left\{ \frac{\mu([0,\phi_i(0)])}{\phi_i(0)^s}, \frac{\mu([\phi_i(0),\phi_i(0)+u])}{u^s} \right\}
	= \min \left\{ d([0,\phi_i(0)]) , d([0,\phi_i^{-1}(u)]) \right\},
	\end{multline*}
	which contradicts the minimality of $ D_1 $.
	
	As a result, 
	$$ 
	D_1 = \min_{1 \leq i \leq n-1} d([0,\phi_i(0)]) = \min_{1 \leq i \leq n-1} \frac{i/l^s}{(i/l)^s} = \min_{1 \leq i \leq n-1} i^{1-s} = 1.
	$$
	
	A similar argument works for the second inequality, but an extra case is needed so we give the details.
	Let $ D_2 = \inf \{ d([x,r]) : 0 \leq x < r \} $, and choose $ x_0 \in \phi_i([0,1])  $ for some $ i \leq n-2 $ (by Equation \ref{dmc}) with $ D_2 = d([x_0,r]) $.
	We claim that $ x_0 = \phi_i(r) $.
	
	If $ x_0 < \phi_i(r) $ we obtain a contradiction in a similar way as in the above argument; setting $ u = \phi_i(r)-x_0 $ and writing $ [x_0,r] $ as the union of $ [\phi_i(r)-u,\phi_i(r)] $ and $ [\phi_i(r),r] $, we can show that 
	$$ 
	D_2 \geq \min \{ d([r-\phi_i^{-1}(u),r]), d([\phi_i(r),r]) \}
	$$ 
	which contradicts minimality. 
	But unlike in the argument for $ D_1 $, we cannot conclude from this that $ x_0 = \phi_i(r) $, because there remains the possibility that $ x_0 \in \phi_i((r,1]) $.
	In this case,
	$$
	d([\phi_i(r),r]) = \frac{\mu([\phi_i(r),x_0])+\mu([x_0,r])}{(r-\phi_i(r))^s}.
	$$
	Notice that $ \mu([\phi_i(r),x_0]) \leq \mu([\phi_i(r),\phi_i(1)]) = \mu([r,1])=0 $, so that
	$$
	d([\phi_i(r),r]) = \frac{\mu([x_0,r])}{(r-\phi_i(r))^s} <  \frac{\mu([x_0,r])}{(r-x_0)^s} = d([x_0,r]) = D_2
	$$
	which also contradicts minimality of $ D_2 $.
	
	As a result,
	$$
	D_2 = \min_{0 \leq i \leq n-2} d([\phi_i(r),r]) = \min_{0 \leq i \leq n-2} \frac{(n-1-i)/l^s}{((n-1-i)/l)^s} = \min_{0 \leq i \leq n-2} (n-1-i)^{1-s} = 1,
	$$
	and this concludes the proof.
\end{proof}

\begin{corollary}
	$ d_{\emph{mc}} \geq 2^{-s} $.
\end{corollary}

\begin{proof}
	The proof is an adaptation of Proposition 3.6 from \cite{Feng} to our setting.
	Let $ U \subset [0,1] $ be a closed interval centered in $ C(n,l) $ with $ U \notin \phi_i([0,1]) $ for all $ i \in I $.
	We will show that $ d(U) \geq 2^{-s} $ by finite induction on the number of level-one basic intervals that intersect $ U $.
	
	First, suppose that $ U = [a,b] $ intersects just one, say $ \phi_j([0,1]) $.
	By our assumption $ U \not\subset \phi_j([0,1]) $, so necessarily $ j = n-1 $ and $ b $ lies in the level-one gap $ (r,1] $.
	Because $ U $ is centered in $ C $, we have $ b-r \leq r-a $, so that
	$$
	d(U) = \frac{\mu([a,r])}{((b-r)+(r-a))^s} \geq \frac{\mu([a,r])}{2^s(r-a)^s} = 2^{-s}d([a,r]) \geq 2^{-s},
	$$
	using Lemma \ref{packlem}.
	
	Now assume the result for all intervals that intersect $ k $ or fewer level-one basic intervals with $ 1 \leq k \leq n-1 $, and suppose that $ U = [a,b] $ intersects $ k+1 $ of these intervals.
	The midpoint $ c = (a+b)/2 $ lies in one of them; for the moment suppose that it does not lie in the leftmost $ \phi_j([0,1]) $.
	Then
	\begin{align*}
	d(U) &= \frac{\mu([a,\phi_j(r)])+\mu([\phi_j(r),2c-\phi_j(r)])+\mu([2c-\phi_j(r),b])}{((\phi_j(r)-a)+(2c-2\phi_j(r))+(b-(2c-\phi_j(r))))^s} \\
	&> \frac{\mu([a,\phi_j(r)])+\mu([\phi_j(r),2c-\phi_j(r)])}{(2(\phi_j(r)-a)+(2c-2\phi_j(r)))^s} \\
	&\geq \frac{\mu([a,\phi_j(r)])+\mu([\phi_j(r),2c-\phi_j(r)])}{2^s(\phi_j(r)-a)^s+(2c-2\phi_j(r))^s} \\
	&\geq \min \left\{ \frac{\mu([a,\phi_j(r)])}{2^s(\phi_j(r)-a)^s}, \frac{\mu([\phi_j(r),2c-\phi_j(r)])}{(2c-2\phi_j(r))^s} \right\} \\
	&= \min \left\{ 2^{-s} d([a,\phi_j(r)]), d([\phi_j(r),2c-\phi_j(r)]) \right\}.
	\end{align*}
	Note that $ d([a,\phi_j(r)]) = d([\phi_j^{-1}(a),r]) \geq 1 $ by the blow-up principle and Lemma \ref{packlem}, and because $ [\phi_j(r),2c-\phi_j(r)] $ is centered in $ C $ and intersects $ \leq k $ level-one basic intervals, $ d([\phi_j(1),2c-\phi_j(1)] \geq 2^{-s} $ by our inductive hypothesis.
	Together, these imply that $ d(U) \geq 2^{-s} $.
	
	The proof for the case that the midpoint does not lie in the rightmost interval is similar, so we omit it.
\end{proof}

\end{document}